\documentclass{amsart}      
\usepackage{amssymb,amsmath,graphicx}
\newtoks\prt 
 
\newtheorem{thma}{Theorem}
 
\newtheorem{thm}{Theorem}[section]
\newtheorem{problem}{Problem}
 
\newtheorem{lemma}[thm]{Lemma} 
\newtheorem{prop}[thm]{Proposition} 
\newtheorem{cor}[thm]{Corollary} 
\newtheorem{example}[thm]{Example}

\theoremstyle{definition} 
 
\newtheorem{remark}[thm]{Remark}

\def\eqn#1$$#2$${\begin{equation}\label#1#2\end{equation}}

\def\C{\mathcal C}

\def\en{\mathbb N}

\def\dist{\operatorname{dist}}

\def \reg {\partial _{\kern1pt\text{reg}}}

\def\ip#1#2{\left\langle#1,#2\right\rangle}

\def\di{\,\mbox{\rm d}}
\def\dh{\widehat{\operatorname{d}}}
\def\clu#1{\operatorname{clust}_{w^*}(#1)}

\newcommand{\norm}[1]{\left\|#1\right\|}

\renewcommand{\Re}{\operatorname{Re}}

\newcommand{\wk}[2][X]{\operatorname{wk}_{#1}\left(#2\right)}
\newcommand{\wck}[2][X]{\operatorname{wck}_{#1}\left(#2\right)}
\newcommand{\wscl}[1]{\overline{#1}^{w^*}}

\newcommand{\bchi}{\mbox{\Large$\chi$}}

\newcommand{\abs}[1]{\left|#1\right|}
\newcommand{\setsep}{;\,}
 
\begin{document}

\title{Measures of weak non-compactness in spaces of nuclear operators}

\author{Jan Hamhalter and Ond\v{r}ej F.K. Kalenda}

\address{Czech Technical University in Prague, Faculty of Electrical Engineering, Department of Mathematics, Technick\'a 2, 166~27 Prague 6, Czech Republic}

\email{hamhalte@math.feld.cvut.cz}

\address{Charles University, Faculty of Mathematics and Physics, Department of Mathematical Analysis,
Sokolovsk\'a 86, 186~75 Praha 8, Czech Republic}

\email{kalenda@karlin.mff.cuni.cz}

\begin{abstract}
We show that in the space of nuclear operators from $\ell^q(\Lambda)$ to $\ell^p(J)$ the two natural ways of measuring weak non-compactness coincide. 
We also provide explicit formulas for these measures. As a consequence the same is proved for preduals of atomic von Neumann algebras.
\end{abstract}


\thanks{Our research was supported in part by the grant
GA\v{C}R 17-00941S}


\maketitle

\section{Introduction}

There are several natural ways how to measure weak non-compactness of bounded subsets of Banach spaces. One possible way was introduced in \cite{deblasi} where it was used to prove stronger and more precise versions of some results on weak compactness, including a fixed-point theorem. Another approach was used to prove a quantitative version of the Krein theorem -- it was done independently in three papers \cite{f-krein,Gr-krein,CMR} using different methods. The second approach inspired a fruitful research, the applications include quantitative versions of several  classical theorems on weak compactness (Eberlein-\v{S}mulyan theorem \cite{AC-jmaa}, Gantmacher theorem \cite{AC-meas},
James compactness theorem \cite{CKS,Gr-James}),
a characterization of subspaces of weakly compactly generated spaces \cite{f-subswcg} or 
a quantitative view on several properties of Banach spaces (Dunford-Pettis property \cite{qdpp}, reciprocal Dunford-Pettis property \cite{rdpp}, Banach-Saks property \cite{BKS-bs} etc.).

It turns out that there are two essentially nonequivalent ways of measuring weak non-compactness -- the one introduced in \cite{deblasi}
and the one used in all the other above-quoted papers. The nonequivalence of the two approaches follows from \cite{tylli-cambridge} as it was explicitly noted in \cite{AC-meas}. The counterexample is constructed as the $c_0$-sum of a sequence of Banach spaces obtained by a suitable renorming of the space $c_0$.
On the other hand, in certain classical spaces the two approaches are equivalent \cite{qdpp}. So, it seems to be an interesting problem whether there is some natural classical space in which the two approaches are not equivalent. 

In the present paper we show that 
the two ways of measuring weak non-com\-pact\-ness coincide in certain spaces of nuclear operators. Let us start by recalling the basic definitions and giving precise formulations of some of the above-mentioned results and problems.

Let $X$ be a Banach space and $A,B\subset X$ two nonempty sets. 
We set
$$\dh(A,B)=\sup\{\dist(a,B)\setsep a\in A\}.$$
This quantity measures how much the set $A$ sticks out from the set $B$ and sometimes it is called the \emph{excess} of $A$ from $B$. Note, that the order of $A$ and $B$ does matter and that $\max\{\dh(A,B),\dh(B,A)\}$ is the Hausdorff distance of the sets $A$ and $B$. 

The quantity $\dh$ is used to define several measures of (weak) non-compactness. Although we focus on weak non-compactness, we begin by the \emph{Hausdorff measure of (norm) non-compactness} which is defined by the formula
$$\chi(A)=\inf\{\dh(A,F)\setsep F\subset X\mbox{ finite}\}=\inf\{\dh(A,K)\setsep K\subset X\mbox{ compact}\}$$
for a bounded set $A\subset X$. It is clear that $\chi(A)=0$ if and only if $A$ is relatively norm compact. 

\emph{De Blasi measure of weak non-compactness} introduced in \cite{deblasi} is defined by 
$$\omega(A)=\inf\{\dh(A,K)\setsep K\subset X\mbox{ weakly compact}\}.$$
It is a natural modification of the Hausdorff measure of non-compactness. Further, $\omega(A)=0$ if and only if $A$ is relatively weakly compact. As remarked in \cite{deblasi} this was proved already by Grothendieck \cite[p. 401]{gro-book}.

Another measure of weak non-compactness inspired by the Banach-Alaoglu theorem is defined by
$$\wk{A}=\dh(\wscl{A},X).$$
Here $\wscl{A}$ is the closure of $A$ in the space $(X^{**},w^*)$, where $X$ is considered to be canonically embedded into its bidual.
It is a direct consequence of the Banach-Alaoglu theorem that a bounded set $A\subset X$ is relatively weakly compact if and only if $\wk{A}=0$. This measure was used, explicitly or implicitly and using different notations, in the above-quoted papers \cite{f-krein,Gr-krein,AC-jmaa,f-subswcg,CKS}. It was established in these papers that it is equivalent to several other measures of weak non-compactness. We will mention and use only one more measure inspired by the Eberlein-\v{S}mulyan theorem and defined by
$$\wck{A}=\sup\{\dist(\clu{x_n},X)\setsep (x_n)\mbox{ is a sequence in }A\},$$
where $\clu{x_n}$ denotes the set of all the weak$^*$-cluster points of the sequence $(x_n)$ in the bidual $X^{**}$. It follows easily from the Eberlein-\v{S}mulyan theorem that $\wck{A}=0$ whenever $A$ is relatively weakly compact. The converse follows from the quantitative version of the Eberlein-\v{S}mulyan theorem proven in \cite{AC-jmaa}. It consists in the inequalities
$$\wck{A}\le \wk{A}\le 2\wck{A}$$
which hold for any bounded subset $A\subset X$.
Further, the following inequalities are easy to check:
$$\wk{A}\le\omega(A)\le\chi(A).$$
The quantities $\wk{\cdot}$ and $\omega(\cdot)$ in general are not equivalent. As mentioned above, this was proved in \cite{tylli-cambridge,AC-meas}. On the other hand, in some classical spaces
the two quantities coincide. Let us recall these results.

The first one concerns the Lebesgue spaces of integrable functions.

\begin{thma} \cite[Lemma 7.4 and Theorem 7.5]{qdpp}
Let $X=L^1(\mu)$ where $\mu$ is any nonegative $\sigma$-additive measure
(not necessarily $\sigma$-finite). Then 
$$\omega(A)=\wk{A}=\wck{A}=\inf\left\{\sup_{f\in A}\int (\abs{f}-c\bchi_E)^+\di\mu \setsep c>0,\mu(E)<\infty \right\} $$
for any bounded set $A\subset X$.
\end{thma} 

The next result concerns a special case of the previous one, namely the space $\ell^1(\Gamma)$. In this case the formula is easier and, moreover,
due to the Schur property the measures of weak non-compactness coincide also with the Hausdorff measure of norm non-compactness.

\begin{thma} \cite[Proposition 7.3]{qdpp}
Let $X=\ell^1(\Gamma)$ for an arbitrary set $\Gamma$. Then
$$\chi(A)=\omega(A)=\wk{A}=\wck{A}=\inf\left\{\sup_{x\in A}\sum_{\gamma\in\Gamma\setminus F}\abs{x_\gamma}\setsep F\subset\Gamma\mbox{ finite} \right\}$$
for any bounded set $A\subset X$.
\end{thma}

The last result concerns the space $c_0(\Gamma)$. Although it was essentially proven in the quoted papers, it is not explicitly formulated
in this form. Therefore we provide a proof, for the sake of completeness.

\begin{thma}
Let $X=c_0(\Gamma)$ for an arbitrary set $\Gamma$.
Then
\begin{multline*}
\omega(A)=\wk A
=\wck A\\=\sup \left\{ \inf\left\{\sup_{\gamma\in\Gamma\setminus F}\liminf_{k\to\infty}\abs{x^k_\gamma}\setsep F\subset\Gamma\mbox{ finite }\right\}\setsep (x^k)\mbox{ is a sequence in }A\right\}
\end{multline*}
for any bounded set $A\subset X$.
\end{thma}

\begin{proof}
The equality $\omega(A)=\wk{A}$ is proved in \cite[Proposition 10.2]{qdpp}, the equality $\wk{A}=\wck{A}$ follows from \cite[Theorem 6.2]{CKS}. It remains to show that the formula gives $\wck{A}$.
To this end first observe that the bidual of $c_0(\Gamma)$ is the space $\ell^\infty(\Gamma)$ and that the weak$^*$ topology on bounded sets coincides with the topology of pointwise convergence.

The inequality `$\ge$' follows from the fact that
for any bounded sequence $(x^k)$ in $c_0(\Gamma)$ we have
$$\dist(\clu{x^k},c_0(\Gamma)) \ge \inf\left\{\sup_{\gamma\in\Gamma\setminus F}\liminf_{k\to\infty}\abs{x^k_\gamma}\setsep F\subset\Gamma\mbox{ finite }\right\}.$$
 Indeed, fix a  bounded sequence $(x^k)$ in $c_0(\Gamma)$ and any $c>\dist(\clu{x^k},c_0(\Gamma))$. So, there is some $y\in\clu{x^k}$ such that $\dist(y,c_0(\Gamma))<c$. It follows that there is a finite set $F\subset \Gamma$ such that $\sup_{\gamma\in \Gamma\setminus F}\abs{y_\gamma}<c$. Since $y_\gamma$ is a cluster point of the sequence $(x^k_\gamma)$, we get $\liminf_{k\to\infty}\abs{x^k_\gamma}\le\abs{y_\gamma}$ for any $\gamma\in\Gamma$. It follows that the quantity on the right-hand side is smaller than $c$. It completes the proof of the inequality `$\ge$'.

To prove the inequality `$\le$' fix any $c<\wck A$. It follows that there is a sequence $(x^k)$ in $A$ such that $\dist(\clu{x^k},c_0(\Gamma))>c$.
Since each $x^k$ has at most countably many nonzero coordinates, there is 
a subsequence $(x^{k_n})$ which pointwise converges on $\Gamma$. Denote the limit $y$. Then $y\in\ell^\infty(\Gamma)$ and the sequence $(x^{k_n})$ weak$^*$ converges to $y$. So, $y\in\clu{x^k}$ and hence $\dist(y,c_0(\Gamma))>c$. Since
$$\begin{aligned}
\dist(y,c_0(\Gamma))&=\inf\left\{\sup_{\gamma\in\Gamma\setminus F}\abs{y_\gamma}\setsep F\subset \Gamma\mbox{ finite}\right\}\\&
=\inf\left\{\sup_{\gamma\in\Gamma\setminus F}\lim_{k\to\infty}\abs{x^k_\gamma}\setsep F\subset \Gamma\mbox{ finite}\right\},\end{aligned}$$
the proof of `$\le$' is completed.
\end{proof}

In view of the three above-mentioned results and of the counterexample of \cite{tylli-cambridge,AC-meas} the following problem seems to be natural and quite interesting.

\begin{problem}
Is there a classical Banach space $X$ in which the measures $\omega(\cdot)$ and $\wk{\cdot}$ are not equivalent?
\end{problem}

Since the notion of a classical Banach space has no precise definition,
we formulate several more concrete questions.

\begin{problem} Let $X=\C(K)$ where $K$ is a compact Hausdorff space. Are the measures $\omega(\cdot)$ and $\wk{\cdot}$ equivalent for bounded subsets of $X$?

Is it true at least for $K=[0,1]$? Is it true at least for $K$ countable?
\end{problem}

This problem was formulated and commented already in \cite[Question 11.1]{qdpp}. Another question concerns a possible non-commutative version 
of Theorem C.

\begin{problem}
Let $X=K(H)$ be the space of compact operators on a Hilbert space (equipped with the operator norm). Are the measures $\omega(\cdot)$ and $\wk{\cdot}$ equivalent for bounded subsets of $X$?
\end{problem}

An even more ambitious problem is the following one, a positive answer would give positive answers to the previous two problems.

\begin{problem}
Let $X$ be a $\C^*$-algebra. Are the measures $\omega(\cdot)$ and $\wk{\cdot}$ equivalent for bounded subsets of $X$?
\end{problem}

The last problem we formulate concerns a non-commutative variant of Theorem~A.

\begin{problem}
Let $X$ be the predual of a von Neumann algebra. Are the measures $\omega(\cdot)$ and $\wk{\cdot}$ equivalent for bounded subsets of $X$?
\end{problem}

In the present paper we prove, among others, a partial positive answer to the last problem. The precise formulation of the results is given in the following section.

\section{Main results}\label{s:main}

Our main results are two theorems on coincidence of measures of weak non-compactness. The first one deals with certain spaces of nuclear operators and the second one with preduals of atomic von Neumann algebras. In this section we recall the necessary definitions and give precise formulations of the results.

Recall that, a linear operator $T:X\to Y$  between Banach spaces is said to be \emph{nuclear} provided there are sequences $(x_n^*)$ in $X^*$ and $(y_n)$ in $Y$ such that $\sum_{n=1}^\infty\norm{x_n^*}\cdot\norm{y_n}<\infty$ such that $Tx=\sum_{n=1}^\infty x_n^*(x)y_n$ for $x\in X$. Further, the nuclear norm $\norm{T}_N$ is defined to be the infimum of the
values $\sum_{n=1}^\infty\norm{x_n^*}\cdot\norm{y_n}<\infty$ over all possible such representations. The space of all the nuclear operators from $X$ to $Y$ equipped with the nuclear norm is denoted by $N(X,Y)$.
The main result on nuclear operators is the following theorem.

\begin{thm}\label{T:main-N} Let $J$ and $\Lambda$ be infinite sets and $p,q\in(1,\infty)$.
\begin{itemize}
	\item[(a)] If $p>q$,  then $N(\ell^q(\Lambda),\ell^p(J))$ is reflexive and hence all the measures of weak non-compactness vanish for any bounded subset of $N(\ell^q(\Lambda),\ell^p(J))$.
	\item[(b)] If $p\le q$, then $Z=N(\ell^q(\Lambda),\ell^p(J))$ is not reflexive and for any bounded set $A\subset Z$ we have
	$$\begin{aligned}	
	\omega(A)=&\wk[Z]{A}=\wck[Z]{A}\\&=\inf\left\{ \sup_{T\in A}\norm{(I-P_C)T(I-Q_D)}_N\setsep C\subset J, D\subset\Lambda \mbox{ finite}\right\},\end{aligned}$$
    where $P_C:\ell^p(J)\to\ell^p(J)$ is the canonical projection annihilating coordinates outside $C$ and $Q_D$ is the analogous projection on $\ell^q(\Lambda)$.
    \item[(c)] For any bounded set $A\subset N(\ell^q(\Lambda),\ell^p(J))$
    we have
    $$\chi(A)\le \inf\left\{ \sup_{T\in A}\norm{T- P_CT Q_D}_N\setsep C\subset J, D\subset\Lambda \mbox{ finite}\right\}\le 2\chi(A).$$
\end{itemize}
\end{thm}

This theorem will be proved in the next section. The assertion (b) can be viewed as a non-commutative version of Theorem~B, as the nuclear operators are, in a sense, a noncommutative version of $\ell^1(\Gamma)$ (this is illustrated by Lemma~\ref{L:Kdiag}(b) below). The assertion (c) is easy and is included 
for the sake of completeness. It also illustrates the difference between the non-commutative and commutative cases. Indeed, unlike the space of nuclear operators, the space $\ell^1(\Gamma)$ has the Schur property (hence $\omega=\chi$ in this space). The comparison of the formulas in (b) and (c) reveals the different nature of the non-commutative case. Furhter, the assertion (c) does not provide a precise formula for $\chi$, only the two inequlaties. This pair of inequalities cannot be easily replaced by an equality by Example~\ref{ex}. We do not know whether the constant $2$ is optimal.

The second main result deals with atomic von Neumann algebras. Recall that
a \emph{von Neumann algebra} is a $*$-subalgebra $M\subset L(H)$ of the space of bounded linear operators on a complex Hilbert space $H$ which is equal to the double commutant of itself. An element $p\in M$ is a \emph{projection} if $p^*=p=p^2$ (i.e., if it is an orthogonal projection
when considered as an operator). A projection $p\in M$ is called \emph{atomic} if the subalgebra $pMp$ has dimension one or, what is the same, if $p$ is an atom in the projection lattice $P(M)$. A von Neumann algebra $M$ is called \emph{atomic} if the unit of $M$ is the sum of atomic projections.

\begin{thm}\label{T:main-vN}
Let $M$ be an atomic von Neumann algebra and $A\subset M_*$ a bounded set.  Then 
$$\omega(A)= \wk[M_*]A=\wck[M_*]A.$$
Moreover, let  $(p_\alpha)_{\alpha\in\Lambda}$ and $(q_j)_{j\in J}$ be two families (possibly but not necessarily the same) of pairwise orthogonal atomic projections, both with sum one.  
Then all the measures of weak non-compactness are equal to 
$$\inf \left\{ \sup_{\varphi\in A} \norm{(1-\sum_{\alpha\in C}p_\alpha)\varphi (1-\sum_{j\in D}q_j)} \setsep C\subset \Lambda,D\subset J\mbox{ finite}\right\} $$
\end{thm}

We see that the formulas in two main results have similar form. In fact, the two results are related -- both of them cover the case of $N(H)$, the space of nuclear operators on a complex Hilbert space (which is the predual of the atomic von Neumann algebra $L(H)$).

Theorem~\ref{T:main-vN} will be proved in the last section as a consequence of Theorem~\ref{T:main-N}.

\section{The case of nuclear operators}

The aim of this section is to prove Theorem~\ref{T:main-N}. To this end we 
will need several results on nuclear operators. Recall that, given Banach spaces $X$ and $Y$, $N(X,Y)$ is the space of nuclear operators from $X$ to $Y$ equipped with the nuclear norm (see Section~\ref{s:main}), $L(X,Y)$ is the space of bounded linear operators from $X$ to $Y$ equipped with the operator norm and $K(X,Y)$ is the subspace of $L(X,Y)$ formed by compact operators. The first result we need is the following one on trace duality.
It is essentially well known as it is clear from the below-given proof.
We will need mainly the assertion (a). For $p\in(1,\infty)$ we will denote by $p^*$ its dual exponent, i.e., the number $p^*\in(1,\infty)$ satisfying $\frac1p+\frac1{p^*}=1$.

\begin{prop}\label{L:dualita} Let $p,q\in(1,\infty)$ and let $J$ and $\Lambda$ be nonempty sets.
\begin{itemize}
	\item[(a)] The dual of $N(\ell^p(J),\ell^q(\Lambda))$ is canonically isometric to $L(\ell^q(\Lambda),\ell^p(J))$, where the duality is given by
	$$\ip{T}{\sum_{n=1}^\infty x_n^*(\cdot) y_n}= \sum_{n=1}^\infty x_n^*(Ty_n)$$
    for $T\in L(\ell^q(\Lambda),\ell^p(J))$ and $\sum_{n=1}^\infty x_n^*(\cdot) y_n\in N(\ell^p(J),\ell^q(\Lambda))$.
  \item[(b)] The dual of $K(\ell^p(J),\ell^q(\Lambda))$ is canonically isometric to $N(\ell^q(\Lambda),\ell^p(J))$, where the duality is given by 
  $$\ip{\sum_{n=1}^\infty x_n^*(\cdot) y_n}{T}= \sum_{n=1}^\infty x_n^*(Ty_n)$$
  for $\sum_{n=1}^\infty x_n^*(\cdot) y_n\in N(\ell^p(J),\ell^q(\Lambda))$ and $T\in K(\ell^q(\Lambda),\ell^p(J))$.
  \item[(c)] The bidual of $K(\ell^p(J),\ell^q(\Lambda))$ is canonically isometric to $L(\ell^p(J),\ell^q(\Lambda))$ and the standard inclusion $K(\ell^p(J),\ell^q(\Lambda))\subset L(\ell^p(J),\ell^q(\Lambda))$ corresponds to the canonical emdedding to the bidual.
\end{itemize}
\end{prop}

\begin{proof} (a) By \cite[Corollary 4.8]{ryan} the space $N(\ell^p(J),\ell^q(\Lambda))$ coincides with the completed projective tensor product $\ell^{p^*}(J)\hat{\otimes}_\pi\ell^q(\Lambda)$. The dual of the projective tensor product is described in \cite[Section 2.2]{ryan}, we use the version given on p. 24. 

(b) By \cite[Corollary 4.13]{ryan} the space $K(\ell^p(J),\ell^q(\Lambda))$ coincides with the completed injective tensor product $\ell^{p^*}(J)\hat{\otimes}_\varepsilon\ell^q(\Lambda)$.
By \cite[p. 67]{ryan}  the dual of this injective tensor product is identified with the space of integral operators 
$I(\ell^q(\Lambda),\ell^p(J))$ equipped with the integral norm (see \cite[p. 62]{ryan}).  By \cite[Corollary 4.17]{ryan} the space $N(\ell^q(\Lambda),\ell^p(J))$ is a closed subspace of $I(\ell^q(\Lambda),\ell^p(J))$ and the nuclear norm coincides with the integral one. Finally, as the spaces $\ell^q(\Lambda)$ and $\ell^p(J)$ are reflexive, an old result of A.Grothendieck
\cite{gro-mem} recalled in \cite[p. 123]{absolutelysumming} implies that $N(\ell^q(\Lambda),\ell^p(J))=I(\ell^q(\Lambda),\ell^p(J))$.

(c) This assertions follows immediately from (a) and (b).
\end{proof}

Using Pitt's theorem \cite[Theorem 4.23]{ryan} the previous proposition implies the following result (cf. \cite[Corollary 4.24]{ryan}):

\begin{cor}\label{c:pitt} Let $p,q\in(1,\infty)$ and let $J$ and $\Lambda$ be infinite sets. The following assertions are equivalent:
\begin{itemize}
	\item[(i)] $K(\ell^p(J),\ell^q(\Lambda))$ is reflexive.
	\item[(ii)] $N(\ell^q(\Lambda),\ell^p(J))$ is reflexive.
	\item[(iii)] $p>q$.
\end{itemize}
\end{cor}

Recall that for $C\subset J$ and $D\subset \Lambda$ we have defined in the statement of Theorem~\ref{T:main-N} projections $P_C$ and $Q_D$. Further,
define $\Phi_{C,D}\in L(K(\ell^p(J),\ell^q(\Lambda)))$ by 
$\Phi_{C,D}(T)=Q_DTP_C$, $T\in K(\ell^p(J),\ell^q(\Lambda))$. Then the following holds:

\begin{lemma}\label{L:PhiCD} For any nonempty sets $C\subset J$ and $D\subset \Lambda$ the following holds:
\begin{itemize}
	\item[(a)] $\Phi_{C,D}$ is a norm-one projection on $K(\ell^p(J),\ell^q(\Lambda))$. If $C$ and $D$ are finite, then it is a finite rank operator.
	\item[(b)] $\Phi_{C,D}^*(T)=P_CTQ_D$ for $T\in N(\ell^q(\Lambda),\ell^p(J))= K(\ell^p(J),\ell^q(\Lambda))^*$.
	\item[(c)] $\Phi_{C,D}^{**}(T)=Q_DTP_C$ for $T\in L(\ell^p(J),\ell^q(\Lambda))= K(\ell^p(J),\ell^q(\Lambda))^{**}$.
\end{itemize}
\end{lemma}

\begin{proof} The assertion (a) is obvious. To show (b) fix $S\in K(\ell^p(J),\ell^q(\Lambda))$ and $T=\sum_{n=1}^\infty x_n^*(\cdot)y_n\in N(\ell^q(\Lambda),\ell^p(J))$. Then
$$\begin{aligned}
\ip{\Phi_{C,D}^*T}{S}&=\ip{T}{\Phi_{C,D}S}=\ip{T}{Q_DSP_C}
=\sum_{n=1}^\infty x_n^*(Q_DSP_Cy_n)\\&=\sum_{n=1}^\infty Q_D^*x_n^*(SP_Cy_n)=\ip{\sum_{n=1}^\infty Q_D^*x_n^*(\cdot)P_Cy_n}{S}\end{aligned}$$
and
$$\sum_{n=1}^\infty Q_D^*x_n^*(x)P_Cy_n=\sum_{n=1}^\infty x_n^*(Q_Dx)P_Cy_n=P_CTQ_Dx$$
for $x\in\ell^q(\Lambda)$.
This completes the proof of (b). The assertion (c) follows by the same computation.
\end{proof}

The next lemma shows how the projections $\Phi_{C,D}$ can be used to approximate identity by finite-rank operators in spaces of operators.
The important case is $1<p\le q<\infty$, but it holds also for $1<q<p<\infty$. In the latter case the assertion (c) is superfluous, as (a) gives a stronger result (due to Corollary~\ref{c:pitt}).

\begin{lemma}\label{L:Kaprox} Let $p,q\in(1,\infty)$ and let $J$ and $\Lambda$ be infinite sets. Let
$$\Sigma=\{ (C,D)\setsep C\subset J \mbox{ finite}, D\subset \Lambda\mbox{ finite}\}$$
be equipped with the partial order defined by $(C,D)\le (C',D')$ if $C\subset C'$ and $D\subset D'$. Then $\Sigma$ is an up-directed set. Moreover, the following assertions hold:
\begin{itemize}
	\item[(a)] The net $(\Phi_{C,D})_{(C,D)\in\Sigma}$ converges to the identity in the strong operator topology of $L(K(\ell^p(J),\ell^q(\Lambda)))$.
	\item[(b)] The net  $(\Phi^*_{C,D})_{(C,D)\in\Sigma}$ converges to the identity in the strong operator topology of $L(N(\ell^q(\Lambda),\ell^p(J)))$.
	\item[(c)] The net  $(\Phi^{**}_{C,D}(T))_{(C,D)\in\Sigma}$ weak$^*$-converges to $T$ for each $T\in L(\ell^p(J),\ell^q(\Lambda))$.
\end{itemize}
\end{lemma}

\begin{proof} It is clear that $\Sigma$ is an up-directed set. So, the limits used in the statements (a)--(c) have sense.

(a) Let $T\in K(\ell^p(J),\ell^q(\Lambda))$ and $\varepsilon>0$. 
Since $T$ is compact, there is a finite set $F\subset T(B_{\ell^p(J)})$ such that $\dh(T(B_{\ell^p(J)}),F)<\frac{\varepsilon}{6}$. Further, find $D\subset \Lambda$ finite such that
$\norm{y-Q_{D}y}<\frac{\varepsilon}{6}$ for $y\in F$. Then $\norm{Q_{D}T-T}\le\frac{\varepsilon}2$. Indeed, given $x\in B_{\ell^p(J)}$ there is $y\in F$ with $\norm{Tx-y}<\frac{\varepsilon}{6}$. Then
$$\norm{(T-Q_{D}T)x}\le \norm{Tx-y}+\norm{y-Q_{D}y}+\norm{Q_{D}(y-Tx)}<\frac{\varepsilon}{6}+\frac{\varepsilon}{6}+\frac{\varepsilon}{6}=\frac{\varepsilon}{2}.$$
Now, $Q_{D}T\in K(\ell^p(J),\ell^q(\Lambda))$, hence the adjoint mapping $(Q_{D}T)^*$ belongs to $K(\ell^{q^*}(\Lambda),\ell^{p^*}(J))$. The above argument can be applied to $(Q_{D}T)^*$ and we can find a finite set $C\subset J$ such that 
$\norm{(Q_{D}T)^*-P_{C}^*(Q_{D}T)^*}\le\frac{\varepsilon}2$. Then clearly $\norm{T-Q_{D}TP_{C}}\le\varepsilon$.
Finally, if $(C',D')\in\Sigma$ with $(C,D)\le(C',D')$, then
$$\begin{aligned}\norm{T-Q_{D'}TP_{C'}}&\le \norm{T-Q_{D}TP_{C}}+\norm{Q_{D}TP_{C}-Q_{D'}TP_{C'}}
\\&=\norm{T-Q_{D}TP_{C}} + \norm{Q_{D'}(Q_{D}TP_{C}-T)P_{C'}}
\le2\varepsilon,\end{aligned}$$
which completes the proof of the convergence.

(b) We will use the description of $\Phi^*_{C,D}$ given in Lemma~\ref{L:PhiCD}(b). Fix $T\in N(\ell^q(\Lambda),\ell^p(J))$
and $\varepsilon>0$. By the definition of nuclear operators there are sequences $(x_n)$ in $\ell^p(J)$ and $(y_n^*)$ in $\ell^{q^*}$ such that $\sum_{n=1}^\infty\norm{x_n}\cdot\norm{y_n^*}<\infty$ and 
$$Ty=\sum_{n=1}^\infty y_n^*(y) x_n, \quad y\in\ell^q(\Lambda).$$
Fix $N\in\en$ such that $\sum_{n>N}\norm{x_n}\cdot\norm{y_n^*}<\varepsilon$. Further, we can find a finite set $C\subset J$ such that $\sum_{n=1}^N \norm{y_n^*}\cdot\norm{x_n-P_C x_n}<\varepsilon$. Finally, we choose a finite set $D\subset\Lambda$
such that $\sum_{n=1}^N \norm{y_n^*-Q_D^* y_n}\cdot\norm{P_C x_n}<\varepsilon$. Then
$$\begin{aligned}\norm{T-P_C T Q_D}_N &= \norm{\sum_{n=1}^\infty y_n^*(\cdot)x_n-\sum_{n=1}^\infty Q_D^*y_n^*(\cdot)P_C x_n}_N
\\&\le \norm{\sum_{n>N}y_n^*(\cdot)x_n}_N + \norm{\sum_{n>N} Q_D^*y_n^*(\cdot)P_C x_n}_N
\\&\qquad + 
\norm{\sum_{n=1}^N y_n^*(\cdot)(x_n-P_C x_n)}_N+\norm{\sum_{n=1}^N (y_n^*- Q_D^*y_n^*)(\cdot)P_C x_n}_N
\\&\le \sum_{n>N}\norm{y_n^*}\norm{x_n} + \sum_{n>N} \norm{Q_D^*y_n^*}\norm{P_C x_n} \\&\qquad+ 
\sum_{n=1}^N \norm{y_n^*}\norm{x_n-P_C x_n}+\sum_{n=1}^N \norm{y_n^*- Q_D^*y_n^*}\norm{P_C x_n}
<4\varepsilon.\end{aligned}$$
Finally, if $(C',D')\in\Sigma$ with $(C,D)\le(C',D')$, then we deduce, similarly as in the proof of (a), that $\norm{Q_{D'}TP_{C'}-T}_N<8\varepsilon$, which completes the proof of the convergence.

(c) This follows easily from (b): Fix  $T\in L(\ell^p(J),\ell^q(\Lambda))$. Then for each $S\in N(\ell^q(\Lambda),\ell^p(J))$
we have
$$\ip{\Phi_{C,D}^{**}T}{S}=\ip{T}{\Phi_{C,D}S}\overset{(C,D)\in\Sigma}{\longrightarrow} \ip{T}{S}.$$
\end{proof}

The assertion (b) of the following lemma says, roughly speaking, that block-diagonal nuclear operators have an $\ell^1$ structure. This will be used as an essential step to prove Theorem~\ref{T:main-N}. The assertion (a) shows that block-diagonal bounded operators have an $\ell^\infty$ structure and, moreover,
block diagonal compact operators have a $c_0$ structure. It is an interesting counterpart of (b) and, moreover, it is used in the proof of (b).
The assertion (a) will be further used in the next section in the proof of Theorem~\ref{T:main-vN}.

\begin{lemma}\label{L:Kdiag} Let $1<p\le q<\infty$. Let $\Gamma$ be a nonempty set, let $(C_\gamma)_{\gamma\in\Gamma}$ be a system of nonempty pairwise disjoint subsets of $J$ and $(D_\gamma)_{\gamma\in\Gamma}$ be a 
	system of finite nonempty pairwise disjoint subsets of $\Lambda$. 
\begin{itemize}
	\item[(a)] 	For each $T\in L(\ell^p(J),\ell^q(\Lambda))$ we have
	 $$\tilde T=\sum_{\gamma\in\Gamma} Q_{D_\gamma}TP_{C_\gamma}\in L(\ell^p(J),\ell^q(\Lambda)),$$
	 where the series converges unconditionally in the weak operator topology. Moreover,
	 $$\norm{\tilde T}=\sup_{\gamma\in\Gamma} \norm{Q_{D_\gamma}TP_{C_\gamma}}\le\norm{T}.$$
     If $T$ is moreover compact, then the series converges unconditionally in the operator norm and $\tilde{T}$ is compact as well.
	\item[(b)] For each $T\in N(\ell^q(\Lambda),\ell^p(\Lambda))$ we have 
	 $$\tilde T=\sum_{\gamma\in\Gamma} P_{C_\gamma}TQ_{D_\gamma}\in N(\ell^q(\Lambda),\ell^p(\Lambda)),$$
	 where the series converges absolutely in the nuclear norm. Moreover,
	 $$\norm{\tilde T}_N=\sum_{\gamma\in\Gamma} \norm{P_{C_\gamma}TQ_{D_\gamma}}_N\le\norm{T}_N.$$
\end{itemize}
	
\end{lemma}

\begin{proof} (a) Let us first suppose that $\Gamma$ is finite. Then the convergence is obvious.
It remains to compute the norm of $\tilde T$. First observe that for any $\gamma\in\Gamma$ we have
$Q_{D_\gamma}\tilde T P_{C_\gamma} = Q_{D_\gamma}TP_{C_\gamma}$, hence $\norm{Q_{D_\gamma}TP_{C_\gamma}}\le\norm{\tilde T}$.
It follows $$\norm{\tilde T}\ge\max_{\gamma\in\Gamma} \norm{Q_{D_\gamma}TP_{C_\gamma}}.$$
Let us prove the converse inequality. Denote  $M=\max_{\gamma\in\Gamma} \norm{Q_{D_\gamma}TP_{C_\gamma}}$  and fix any $x\in\ell^p(J)$ with $\norm{x}\le 1$. Then
$$\begin{aligned}\norm{\sum_{\gamma\in\Gamma} Q_{D_\gamma}TP_{C_\gamma}(x)}
&=\left(\sum_{\gamma\in\Gamma} \norm{Q_{D_\gamma}TP_{C_\gamma}(x)}^q\right)^{1/q}
=\left(\sum_{\gamma\in\Gamma} \norm{Q_{D_\gamma}TP_{C_\gamma}P_{C_\gamma}(x)}^q\right)^{1/q}
\\&\le M \left(\sum_{\gamma\in\Gamma} \norm{P_{C_\gamma}x}^q\right)^{1/q}
\le M \left(\sum_{\gamma\in\Gamma} \norm{P_{C_\gamma}x}^p\right)^{1/q}
\le M\end{aligned}.$$
Note that the assumption $p\le q$ was used in the second inequality on the second line of the computation. We have showed that $\norm{\tilde{T}}= M$. Since clearly $M\le\norm{T}$, we deduce that $\norm{\tilde{T}}\le\norm{T}$.

Next suppose that $\Gamma$ is general (possibly infinite). Similarly as above set $M=\sup_{\gamma\in\Gamma} \norm{Q_{D_\gamma}TP_{C_\gamma}}$. Clearly $M\le\norm{T}$. Fix any $x\in\ell^p(J)$. For any finite set $F\subset \Gamma$ we have (by the same computation as above)
$$\norm{\sum_{\gamma\in F} Q_{D_\gamma}TP_{C_\gamma}(x)}
\le M \left(\sum_{\gamma\in F} \norm{P_{C_\gamma}x}^p\right)^{1/q}.$$
Since, $\sum_{\gamma\in \Gamma} \norm{P_{C_\gamma}x}^p\le\norm{x}^p$, 
the Bolzano-Cauchy condition shows that the series $\sum_{\gamma\in\Gamma}Q_{D_\gamma}TP_{C_\gamma}(x)$ converges unconditionally in the norm. If we denote the limit $\tilde{T}(x)$, it is clear that $\tilde{T}$ is a linear operator and $\norm{\title{T}}\le M$. The converse inequality is obvious.

Finally, let us assume that $T$ is moreover compact. 
We will show the unconditional convergence by verifying the Bolzano-Cauchy condition. Fix $\varepsilon>0$. By Lemma~\ref{L:Kaprox}(a) there are finite sets $C\subset J$ and $D\subset \Lambda$ such that $\norm{T-Q_DTP_C}\le\varepsilon$.
Let $H\subset \Gamma$ be finite such that $C\gamma\cap C=\emptyset$ and $D_\gamma\cap D=\emptyset$ for $\gamma\in H$.
Then for any $\gamma\in H$ we have
$$\norm{Q_{D_\gamma}TP_{C_\gamma}}=\norm{Q_{D_\gamma}(T-Q_DTP_C)P_{C_\gamma}}\le\varepsilon,$$
thus
$$\norm{\sum_{\gamma\in H} Q_{D_\gamma}TP_{C_\gamma}}\le\varepsilon.$$
This completes the proof of the Bolzano-Cauchy condition and hence the proof of the unconditional convergence of the series.

(b) Let us first suppose that $\Gamma$ is finite. Then the convergence is obvious and $\tilde T$ is clearly a nuclear operator.
It remains to compute the norm of $\tilde T$. By the triangle inequality we get
$$\norm{\tilde T}_N\le \sum_{\gamma\in \Gamma}\norm{P_{C_\gamma}TQ_{D_\gamma}}_N.$$
Let us show the converse inequality. To this end we will use (a) and Proposition~\ref{L:dualita}. Fix $\varepsilon>0$.
Denote by $M$ the cardinality of $\Gamma$. For each $\gamma\in\Gamma$ we can find $S_\gamma\in K(\ell^p(J),\ell^q(\Lambda))$ such that $\norm{S_\gamma}\le 1$ and $$\Re \ip{P_{C_\gamma}TQ_{D_\gamma}}{S_\gamma}>\norm{P_{C_\gamma}TQ_{D_\gamma}}_N-\frac\varepsilon M.$$
Let 
$$S=\sum_{\gamma\in\Gamma} Q_{D_\gamma}S_\gamma P_{C_\gamma}.$$
Then $S$ is a compact operator and by (a) we deduce $\norm{S}\le 1$. Thus
$$\begin{aligned}\norm{\tilde T}_N&\ge\Re\ip{\tilde T}{S}= \sum_{\gamma,\delta\in\Gamma} \Re\ip{P_{C_\gamma}TQ_{D_\gamma}}{P_{C_\delta}S_\delta Q_{D_\delta}}\\&=\sum_{\gamma,\delta\in\Gamma}\Re\ip{\Phi_{C_\gamma,D_\gamma}^*(T)}{\Phi_{C_\delta,Q_\delta}(S_\delta)}
=\sum_{\gamma,\delta\in\Gamma}\Re\ip{\Phi_{C_\delta,Q_\delta}^*(\Phi_{C_\gamma,D_\gamma}^*(T))}{S_\delta}
\\&=\sum_{\gamma\in\Gamma} \Re\ip{\Phi_{C_\gamma,D_\gamma}^*(T)}{S_\gamma}=\sum_{\gamma\in\Gamma} \Re\ip{P_{C_\gamma}TQ_{D_\gamma}}{S_\gamma}\\&>\sum_{\gamma\in\Gamma}\norm{P_{C_\gamma}TQ_{D_\gamma}}_N-\varepsilon.\end{aligned}$$
This completes the proof of the second inequality. Moreover, a similar computation
$$\begin{aligned}\norm{T}_N&\ge \Re\ip{T}{S}=\sum_{\gamma\in\Gamma} \Re\ip{T}{P_{C_\gamma}S_\gamma Q_{D_\gamma}}
=\sum_{\gamma\in\Gamma}\Re\ip{T}{\Phi_{C_\gamma,D_\gamma}(S_\gamma)}\\&=
\sum_{\gamma\in\Gamma}\Re\ip{\Phi_{C_\gamma,D_\gamma}^*(T)}{S_\gamma}=\sum_{\gamma\in\Gamma} \Re\ip{P_{C_\gamma}TQ_{D_\gamma}}{S_\gamma}\\&>\sum_{\gamma\in\Gamma}\norm{P_{C_\gamma}TQ_{D_\gamma}}_N-\varepsilon.´
\end{aligned}$$
shows that $\norm{\tilde T}_N\le\norm{T}_N$.

To prove the statement for a general $\Gamma$ it is enough to prove the absolute convergence of the series. The rest then follows easily. We will show the absolute convergence by verifying the Bolzano-Cauchy condition. Fix $\varepsilon>0$. By Lemma~\ref{L:Kaprox}(b) there are finite sets $C\subset J$ and $D\subset \Lambda$ such that $\norm{T-P_CTQ_D}_N\le\varepsilon$.
Let $H\subset \Gamma$ be finite such that $C\gamma\cap C=\emptyset$ and $D_\gamma\cap D=\emptyset$ for $\gamma\in H$.
Then
$$\sum_{\gamma\in H}\norm{P_{C_\gamma}TQ_{D_\gamma}}_N=\sum_{\gamma\in H}\norm{P_{C_\gamma}(T-P_CTQ_D)Q_{D_\gamma}}_N\le \norm{T-P_CQD_D}_N\le\varepsilon.$$
This completes the proof of the Bolzano-Cauchy condition and hence the proof of the absolute convergence of the series.
\end{proof}

The next lemma identifies certain reflexive subspaces of spaces of operators and will be used to compute the quantity $\omega(\cdot)$.

\begin{lemma}\label{L:reflexive} Let $p,q\in(1,\infty)$ and let $C\subset J$, $D\subset \Lambda$ be finite sets. Then
\begin{itemize}
	\item[(a)] $K_{C,D}(\ell^p(J),\ell^q(\Lambda))=\{T\in K(\ell^p(J),\ell^q(\Lambda))\setsep (I-Q_D)T(I-P_C)=0\}$
is a reflexive subspace of $K(\ell^p(J),\ell^q(\Lambda))$.
\item[(b)] $N_{C,D}(\ell^q(\Lambda),\ell^p(J))=\{T\in N(\ell^q(\Lambda),\ell^p(J))\setsep (I-P_C)T(I-Q_D)=0\}$
is a reflexive subspace of $K(\ell^p(J),\ell^q(\Lambda))$.
\end{itemize}
\end{lemma}

\begin{proof}
(a) Consider the operators $\Phi_{C,D}$, $\Phi_{J\setminus C,D}$, $\Phi_{C,\Lambda\setminus D}$ and $\Phi_{J\setminus C,\Lambda\setminus D}$. All these four operators are are norm-one projections and $K(\ell^p(J),\ell^q(\Lambda))$
is the direct sum of their ranges. Further, the subspace $K_{C,D}(\ell^p(J),\ell^q(\Lambda))$ is the sum of the ranges of first three of these projections. So, it is enough to observe that the three ranges are reflexive. The range of $\Phi_{C,D}$ is finite-dimensional. The range of $\Phi_{J\setminus C,D}$ is canonically isometric to
$L(\ell^p(J\setminus C),\ell^q(D))$, which is isomorphic to  $\ell^{p^*}(J\setminus C)^m$, where $m$ is the cardinality of $D$, hence it is reflexive.  Finally, the range of $\Phi_{C,\Lambda\setminus D}$ is canonically isometric to 
$L(\ell^p(C),\ell^q(\Lambda\setminus D))$, which is isomorphic to $\ell^q(\Lambda\setminus D)^n$, where $n$ is the cardinality of $C$, hence it is reflexive as well.
 
(b) This can be proved similarly as (a) or, alternatively, it can be deduced from (a) using the observation that $N_{C,D}(\ell^q(\Lambda),\ell^p(J))$, being the range of the projection $(I-\Phi_{J\setminus C,\Lambda\setminus D})^*$, is canonically isomorphic to the dual of $K_{C,D}(\ell^p(J),\ell^q(\Lambda))$, which is the range of the projection $I-\Phi_{J\setminus C,\Lambda\setminus D}$.
\end{proof}

The following easy abstract lemma will be used together with the previous one.

\begin{lemma}\label{L:omegaref} Let $X$ be a Banach space, $Y\subset X$ a reflexive subspace and $A\subset X$ a bounded set. Then
$\omega(A)\le\dh(A,Y)$.
\end{lemma}

\begin{proof} Let $R=\sup\{\norm{x}\setsep x\in A\}$. Given $x\in A$ and $y\in Y$ with $\norm{y}>2R$, then
$$\norm{x-y}\ge\norm{y}-\norm{x} > R \ge\norm{x}=\norm{x-0},$$
so $\dist(x,Y)=\dist(x,Y\cap RB_X)$. Since $Y\cap RB_X$ is weakly compact, we deduce that
$$\omega(A)\le\dh(A,Y\cap RB_X)=\dh(A,Y),$$
which completes the proof.
\end{proof}

The next lemma compares measures of weak non-compactness in a Banach space and in a $1$-complemented subspace.

\begin{lemma}\label{L:1-compl} Let $X$ be a Banach space and $Y\subset X$ a $1$-complemented subspace of $X$. Let $P:X\to Y$ be a norm-one projection of $X$ onto $Y$.
\begin{itemize}
\item[(a)] $\wk[Y]{P(A)}\le\wk A$, $\wck[Y]{P(A)}\le\wck A$  and  $\omega_Y(P(A))\le\omega_X(A)$ for any bounded set $A\subset X$.
\item[(b)] $\wk[Y]A=\wk A$, $\wck[Y]A=\wck A$  and  $\omega_Y(A)=\omega_X(A)$ for any bounded set $A\subset Y$.
\end{itemize}
\end{lemma}

\begin{proof} (a) We start by observing that
\begin{equation}\label{eq:compl1}
\dist(P^{**}x^{**},Y)\le\dist(x^{**},X)\mbox{ for any }x^{**}\in X^{**}.
\end{equation}
Indeed, given $x^{**}\in X^{**}$ and $x\in X$ 
we have
$$\norm{x^{**}-x}\ge\norm{P^{**}(x^{**}-x)}=\norm{P^{**}x^{**}-Px}.$$

Let us continue by proving the first inequality. Since $A$ is bounded and $P^{**}$ is weak$^*$-to-weak$^*$ continuous, we see that $P^{**}(\wscl{A})=\wscl{P(A)}$. Therefore, given any $y^{**}\in \wscl{P(A)}$, there is $x^{**}\in \wscl{A}$ with $P^{**}(x^{**})=y^{**}$. By \eqref{eq:compl1} we see that $\dist(y^{**},Y)\le \dist(x^{**},X)$. Since $y^{**}$ was arbitrary, we deduce $\wk[Y]{P(A)}\le\wk A$.

To prove the second inequality fix a sequence $(y_k)$ in $P(A)$. We can find a sequence $(x_k)$ in $A$ with $y_k=Px_k$ for each $k\in\en$. For any weak$^*$-cluster point $x^{**}$ of $(x_k)$ its image $P^{**}x^{**}$ is a weak$^*$-cluster point of $(y_k)$ and, by \eqref{eq:compl1}, we see that $\dist(P^{**}x^{**},Y)\le \dist(x^{**},X)$. 
Therefore $\dist(\clu{y_k},Y)\le \dist(x^{**},X)$. 
Since $x^{**}$ was arbitrary, we deduce
$$\dist(\clu{y_k},Y)\le \dist(\clu{x_k},X)\le\wck{A}.$$
Since this holds for any sequence $(y_k)$ in $P(A)$, we conclude 
$\wck[Y]{P(A)}\le\wck A$.

Finally, fix any $c>\omega_X(A)$. It follows that there is a weakly compact set $K\subset X$ with $\dh(A,K)<c$. Then $P(K)$ is a weakly compact subset of $Y$. Moreover, given $a\in A$ and $x\in K$, we have
$$\norm{a-x}\ge\norm{P(a-x)}=\norm{Pa-Px}\ge \dist(Pa,P(K)).$$
It follows that $\dh(A,K)\ge\dh(P(A),P(K))$, thus $\omega_Y(P(A))<c$. This completes the proof.

(b)  The inequalities `$\geq$' in all the three cases are obvious and holds also without the complementability assumption.
The convese inequalities follow from (a).
\end{proof}

The last ingredient is the following result from \cite{qdpp}.

\begin{lemma}\label{L:l1suma} \ \cite[Lemma 7.2(iii)]{qdpp} 
Let $(X_\gamma)_{\gamma\in\Gamma}$ be a family of reflexive Banach spaces and let 
$X=\left(\bigoplus_{\gamma\in \Gamma} X_\gamma\right)_{\ell^1}$. Then for any nonempty bounded set $A\subset X$ we have
$$\omega(A)=\wk A=\wck A=\inf\left\{\sup_{x\in A} \sum_{\gamma\in\Gamma\setminus F}\norm{x_\gamma}
\setsep F\subset \Gamma\mbox{ finite}\right\}.$$
\end{lemma}

Now we are ready to give the proof of the first main result.

\medskip

\noindent{\bf Proof of Theorem~\ref{T:main-N}.} 
The assertion (a) follows from Corollary~\ref{c:pitt}.

(b) Suppose $p\le q$. The inequalities $\wck[Z]{A}\le\wk[Z]{A}\le\omega(A)$ are obvious. 

Further, fix any $C\subset J$ and $D\subset\Lambda$ finite. Set $Y=N_{C,D}(\ell^q(\Lambda),\ell^p(J))$ (see Lemma~\ref{L:reflexive}(b)). By the quoted lemma the subspace $Y$ is reflexive. By Lemma~\ref{L:omegaref} we deduce
that $\omega(A)\le\dh(A,Y)$. Since it is clear that
$$\dh(A,Y)\le \sup_{T\in A}\norm{(I-P_C)T(I-Q_D)},$$ we deduce that
$$\omega(A)\le\inf\left\{ \sup_{T\in A}\norm{(I-P_C)T(I-Q_D)}_N\setsep C\subset J, D\subset\Lambda \mbox{ finite}\right\}.$$

To complete the proof  it suffices to show that
$$\inf\left\{ \sup_{T\in A}\norm{(I-P_C)T(I-Q_D)}_N\setsep C\subset J, D\subset\Lambda \mbox{ finite}\right\}\le\wck[Z]{A}.$$
If the left-hand side is zero, it is trivial. Suppose that the left hand-side is strictly positive and choose any numbers $c$ such that
$$0<c<\inf\left\{ \sup_{T\in A}\norm{(I-P_C)T(I-Q_D)}_N\setsep C\subset J, D\subset\Lambda \mbox{ finite}\right\}$$
and an arbitrary $\varepsilon\in(0,c)$.
By induction we can construct a sequence $(T_n)$ in $A$ and a sequence of finite sets
$C'_n\subset J$ and $D'_n\subset \Lambda$ with the following properties:
\begin{itemize}
	\item $C'_1=D'_1=\emptyset$,
	\item $\norm{(I-P_{C'_n})T_n(I-Q_{D'_n})}_N>c$ for $n\in\en$,
	\item $C'_{n+1}\supset C'_n$ and $D'_{n+1}\supset D'_n$ for $n\in\en$,
	\item $\norm{T_k-P_{C'_{n+1}}T_kQ_{D'_{n+1}}}_N<\varepsilon$ for  $k,n\in\en$, $n\ge k$.
\end{itemize}
The construction can be done easily by an iterated use of the assumptions and Lemma~\ref{L:Kaprox}(b).

For $n\in \en$ set $C_n=C'_{n+1}\setminus C_n$ and $D_n=D'_{n+1}\setminus D'_n$. Let $\Psi:T\to\tilde T$ be the mapping provided by Lemma~\ref{L:Kdiag}(b). By the quoted lemma we see that $\Psi$ is a norm-one projection. Hence, by  Lemma~\ref{L:1-compl} we see that $\wck[Z]{A}\ge\wck[\Psi(Z)]{\Psi(A)}$. Moreover, by Lemma~\ref{L:Kdiag}(b) the range $\Psi(Z)$ is canonically isometric to the $\ell^1$-sum of the finite-dimensional spaces $P_{C_n}ZQ_{D_n}$.
So, we deduce by Lemma~\ref{L:l1suma} that
$$\wck[\Psi(Z)]{\Psi(A)}
\ge\inf_{m\in\en} \sup_{k\in\en} \sum_{n\ge m}\norm{P_{C_n}T_k Q_{D_n}}_N\ge\inf_{m\in\en}\norm{P_{C_m}T_m Q_{D_m}}_N.
$$
For any $m\in\en$ we have
$$\begin{aligned}
\norm{P_{C_m}T_m Q_{D_m}}_N
&\ge\norm{(I-P_{C'_m})T_m(I-Q_{D'_m})}_N\\&\qquad\qquad-\norm{P_{C_m}T_m Q_{D_m}-(I-P_{C'_m})T_m(I-Q_{D'_m})}_N
\\&>c-\norm{(I-P_{C'_m})(P_{C'_{m+1}}T_m Q_{D'_{m+1}}-T_m)(I-Q_{D'_m})}_N
\\&\ge c-\norm{P_{C'_{m+1}}T_m Q_{D'_{m+1}}-T_m}_N>c-\varepsilon.
\end{aligned}$$
Hence, $\wck[\Psi(Z)]{\Psi(A)}\ge c-\varepsilon$ and therefore $\wck[Z]A\ge c-\varepsilon$.
Since $c$ and $\varepsilon$ were arbitrary, this completes the proof.

(c) Suppose $A\subset N(\ell^q(\Lambda),\ell^p(J))$ is a nonempty bounded set. Fix any finite sets $C\subset J$ and $D\subset\Lambda$. Then
$$B=\{P_CTQ_D\setsep T\in A\}$$ 
is a bounded subset of a finite-dimensional subspace, so it is relatively norm-compact. Therefore
$$\chi(A)\le\dh(A,B)\le\sup_{T\in A}\norm{T-P_CTQ_D}_N.$$
This completes the proof of the first inequality.

To prove the second inequality, fix any $\varepsilon>0$. Then there is a finite set $F\subset N(\ell^q(\Lambda),\ell^p(J))$ with $\dh(A,F)<\chi(A)+\varepsilon$. By Lemma~\ref{L:PhiCD}(b) there are finite sets $C\subset J$ and $D\subset\Lambda$ such that $\norm{S-P_CSQ_D}_N<\varepsilon$ for each $S\in F$. Fix any $T\in A$. Then there is $S\in F$ with $\norm{T-S}_N<\chi(A)+\varepsilon$. Therefore
$$\begin{aligned}
\norm{T-P_CTQ_D}_N &\le \norm{T-S}_N+\norm{S-P_CSQ_D}_N+\norm{P_C(S-T)Q_D}_N\\&\le\chi(A)+\varepsilon+\varepsilon+\chi(A)+\varepsilon=2\chi(A)+3\varepsilon.
\end{aligned}
$$
Since $\varepsilon>0$ is arbitrary, the second inequality follows.
\qed

\begin{example}\label{ex}
There are two sets $A,B\subset N(\ell^2)$ such that
$$\begin{gathered}
\chi(A)=\inf\left\{ \sup_{T\in A}\norm{T- P_CT Q_D}_N\setsep C,D\subset \en \mbox{ finite}\right\}=1,\\
\chi(B)<\inf\left\{ \sup_{T\in B}\norm{T- P_CT Q_D}_N\setsep C,D\subset \en \mbox{ finite}\right\}.
\end{gathered}$$
\end{example}

\begin{proof}
(a) For $n\in\en$ define the operator
$$T_n(x)=x_ne_n,\quad x=(x_k)\in\ell^2$$
and set $A=\{T_n\setsep n\in\en\}$. It is clear that $\norm{T_n}_N=1$ for each $n\in\en$ (cf. \cite[Proposition 2.1]{ryan}, recall that by \cite[Corollary 4.8]{ryan} the space $N(\ell^2)$ coincides with the completed projective tensor product $\ell^{2}\hat{\otimes}_\pi\ell^2$). Therefore it is obvious that
$$\inf\left\{ \sup_{T\in A}\norm{T- P_CT Q_D}_N\setsep C,D\subset \en \mbox{ finite}\right\}=1.$$
It remains to show that $\chi(A)\ge 1$ (the converse inequality follows from Theorem~\ref{T:main-N}(c).) To this end fix any finite set $F\subset N(\ell^2)$ and $\varepsilon>0$. By Lemma~\ref{L:Kaprox}(b) there are finite sets $C,D\subset N$ such that $\norm{S-P_CSQ_D}_N<\varepsilon$ (and, a fortiori, $\norm{(I-P_C)S(I-Q_D)}_N<\varepsilon$) for each $S\in F$. Let $n\in\en\setminus(C\cup D)$. Then for any $S\in F$ we have
$$\begin{aligned}
\norm{T_n-S}_N&\ge \norm{(I-P_C)(T_n-S)(I-Q_D)}_N=
 \norm{T_n-(I-P_C)S(I-Q_D)}_N\\&\ge \norm{T_n}_N- \norm{(I-P_C)S(I-Q_D)}_N
 >1-\varepsilon.\end{aligned}$$
 It follows $\chi(A)\ge 1-\varepsilon$. Since $\varepsilon>0$ is arbitrary, we deduce $\chi(A)\ge 1$.
 
 (b) For $n\in\en$ define the operator
$$U_n(x)=(x_1+x_n)(e_1+e_n),\quad x=(x_k)\in\ell^2$$
 and set $B=\{U_n\setsep n\in\en\}\cup\{10T_1\}$, where $T_1$ is the operator from (a). Since $\norm{U_n}_N=2$ for each $n\in\en$ (this follows again from \cite[Proposition 2.1]{ryan}),
 we have $\dh(B,\{0,10T_1\})\le 2$, hence $\chi(B)\le 2$.
 
Let $C,D\subset \en$ be two finite sets. If $1\notin C\cap D$, then $P_CT_1Q_D=0$, hence
$$\sup_{T\in B}\norm{T- P_CT Q_D}_N\ge \norm{10T_1}_N=10.$$
Suppose that $1\in C\cap D$. Fix $n\in\en$ with $n\notin C\cup D$. 
Then
$$\sup_{T\in B}\norm{T- P_CT Q_D}_N\ge \norm{U_n-P_CU_n Q_D}_N.$$
Set $U=U_n-P_CU_n Q_D$ and compute its nuclear norm. We have
$$
Ux=x_ne_1 + (x_1+x_n)e_n, \quad x=(x_k)\in\ell^2.
$$
Then $U$ is selfadjoint, hence 
$$U^*U x=U^2 x=(x_1+x_n)e_1 +(x_1+2x_n)e_n, \quad x=(x_k)\in\ell^2.$$
An easy computation shows that the eigenvalues of $U^*U$ are $0$, $\frac{3+\sqrt5}{2}$ and $\frac{3-\sqrt5}{2}$. Therefore (using for example \cite[the first paragraph on p. 151, Definition on p. 152 and Lemma 16.13]{meisevogt}) we get
$$\begin{aligned}
\norm{U}_N&=\sqrt{\frac{3+\sqrt5}{2}}+\sqrt{\frac{3-\sqrt5}{2}}=\frac{\sqrt{5+2\sqrt5+1}+\sqrt{5-2\sqrt5+1}}{2}
\\&=\frac{\sqrt5+1+\sqrt5-1}{2}=\sqrt5. \end{aligned}$$
It follows that
$$\inf\left\{ \sup_{T\in B}\norm{T- P_CT Q_D}_N\setsep C,D\subset \en \mbox{ finite}\right\}\ge\sqrt5>2\ge\chi(B),$$
which completes the proof.
\end{proof}

\section{Preduals of atomic von Neumann algebras}

In this section we provide a proof of Theorem~\ref{T:main-vN}.
The proof will be done using Theorem~\ref{T:main-N} and some facts on von Neumann algebras. The mere coincidence of the measures of weak non-compactness easily follows from Theorem~\ref{T:main-N} using known results.
Indeed, if $M\subset L(H)$ is an atomic von Neumann algebra, by \cite[Theorem 4.2.2]{blackadar} there is a norm-one projection $P:L(H)\to M$ which is also weak$^*$-to-weak$^*$ continuous. It follows that the predual of $M$ is $1$-complemented in $N(H)$, which is the predual of $L(H)$. So, the equality of measures of weak non-compactness follows from Theorem~\ref{T:main-N} and Lemma~\ref{L:1-compl}.

If we wish to prove not only the coincidence of the measures of weak non-compactness, but also the formula given in the statement of Theorem~\ref{T:main-vN}, we need a more detailed description of $P$ and a representation of $M$. Such a representation is mentioned (without proof) 
in \cite[p. 3]{Akemann-Anderson}. We give a complete elementary proof
for the sake of completeness and, further, in order to obtain the formula easily.

\begin{prop}\label{p:rozklad} Let $H$ be a complex Hilbert space and $M\subset L(H)$ an atomic von Neumann algebra separating points of $H$. Let $(p_\alpha)_{\alpha\in\Lambda}$ be a family of pairwise orthogonal atomic projections in $M$ with sum equal to the unit (i.e., to the identity operator). For any $C\subset \Lambda$ denote $p_C=\sum_{\alpha\in C} p_\alpha$. Then there is a partition $(Z_\gamma)_{\gamma\in\Gamma}$ of $\Lambda$ to nonempty subsets such that
\begin{itemize}
\item[(i)] $p_{Z_\gamma}$ belongs to the center of $M$ for each $\gamma\in\Gamma$,
\item[(ii)] $M=\left\{ T\in L(H)\setsep T=\sum_{\gamma\in\Gamma} p_{Z_\gamma}Tp_{Z_\gamma} \right\},$ where the sum is taken in the weak operator topology.
\end{itemize}
Moreover, the family of projections $\{p_{Z_\gamma}\setsep\gamma\in\Gamma\}$ does not depend on the particular choice of the family $(p_\alpha)_{\alpha\in\Lambda}$.
\end{prop}

\begin{proof}
Since $M$ separates points of $H$, any atomic projection $p\in M$ is, when considered as an operator on $H$, an orthogonal projection to a one-dimensional subspace. So, there is an orthonormal basis $(e_\alpha)_{\alpha\in\Lambda}$ of $H$ such that, for each $\alpha\in\Lambda$, $p_\alpha$ is the projection onto the linear span of $e_\alpha$. 
For any $\alpha\in \Lambda$ set
$$Z_\alpha=\bigcap\{C\subset\Lambda\setsep \alpha\in C\ \&\ p_C\mbox{ belongs to the center of }M\}.$$
It is clear that $Z_\alpha$ is a well defined set containing $\alpha$ (hence nonempty) and that $p_{Z_\alpha}$ belongs to the center for each $\alpha\in\Lambda$. Moreover, for $\alpha,\beta\in\Lambda$ either $Z_\alpha=Z_\beta$ or $Z_\alpha\cap Z_\beta=\emptyset$. So, we have a decomposition 
$$\Lambda=\bigcup_{\gamma\in\Gamma}Z_\gamma$$
to pairwise disjoint nonempty sets such that for each $\gamma\in Z_\gamma$ the projection $p_{Z_\gamma}$ belongs to the center, but for $\emptyset\ne C\subsetneqq Z_\gamma$ the projection $p_C$ does not belong to the center.
It follows that (i) is satisfied.
It is clear that for any $T\in M$ we have
$$T=\sum_{\gamma\in\Gamma}p_{Z_\gamma}T=\sum_{\gamma\in\Gamma}p_{Z_\gamma}Tp_{Z_\gamma},$$
the sum being taken in the weak operator topology. Hence, the inclusion `$\subset$' of the equality from (ii) is valid.

To prove the converse inclusion we will use the fact that $M$ is a von Neumann algebra, hence $M$ equals its double-commutant $M''$. Therefore, it is enough to show that the commutant of $M$ equals
\begin{equation}\label{eq:M'}
M'=\left\{\sum_{\gamma\in\Gamma}\lambda_\gamma p_{Z_\gamma}\setsep \sup_{\gamma\in\Gamma}\abs{\lambda_\gamma}<\infty\right\}.\end{equation}
The inclusion `$\supset$' is obvious, let us show the converse. Let $T\in M'$. Given $\alpha\in\Lambda$, we have $p_\alpha\in M$ so $p_\alpha T=Tp_\alpha$. It follows that $T$ is a diagonal operator, i.e., 
$$T(x)=\sum_{\alpha\in\Lambda} c_\alpha \ip{x}{e_\alpha}e_\alpha, \qquad x\in H,$$
for a bounded set of coefficients $(c_\alpha)$. It remains to show that $c_\alpha=c_\beta$ if $\alpha,\beta\in Z_\gamma$ for some $\gamma$. 
So, fix $\gamma\in\Gamma$ and $\alpha\in Z_\gamma$. Set $C=\{\beta\in Z_\gamma\setsep c_\beta=c_\alpha\}$. If $C\subsetneqq Z_\gamma$, then
the projection $p_C$ does not belong to the center of $M$. It follows that there is $S\in M$ with $p_CS\ne Sp_C$. Hence at least one of the spaces
$p_CH$ and $p_{\Lambda\setminus C}H$ is not invariant for $S$.
On the other hand, $p_{Z_\gamma}H$ and $p_{\Lambda\setminus Z_\gamma}H$ are invariant for $S$ (as $p_{Z_\gamma}$ belongs to the center). It follows that there are two distinct points $\beta_1,\beta_2\in Z_\gamma$ such that exactly one of them belongs to $C$ and $\ip{Se_{\beta_1}}{e_{\beta_2}}\ne0$.
Since $T\in M'$, we have $ST=TS$, hence
$$  c_{\beta_1} \ip{Se_{\beta_1}}{e_{\beta_2}} = \ip{S(c_{\beta_1}e_{\beta_1})}{e_{\beta_2}}=\ip{STe_{\beta_1}}{e_{\beta_2}}= \ip{TSe_{\beta_1}}{e_{\beta_2}}=c_{\beta_2} \ip{Se_{\beta_1}}{e_{\beta_2}}.
$$
Since  $\ip{Se_{\beta_1}}{e_{\beta_2}}\ne0$, we deduce $c_{\beta_1}
=c_{\beta_2}$, a contradiction. This completes the proof of (ii).

It remains to show that the family $\{p_{Z_\gamma}\setsep\gamma\in\Gamma\}$ does not depend on the particular choice of the family $(p_\alpha)_{\alpha\in\Lambda}$. To see this observe that the projections $p_{Z_\gamma}$ are precisely the minimal central projections. Indeed, let $q\in M$ be a central projection. It follows from the description of $M'$ in \eqref{eq:M'} that there is a subset $\Gamma'\subset\Gamma$ such that $q=\sum_{\gamma\in\Gamma'}p_{Z_\gamma}$. So, it is clear that the minimal central projections are precisely the projections $p_{Z_\gamma}$.
\end{proof}

Now we are ready to prove the theorem.

\medskip

\noindent {\bf Proof of Theorem~\ref{T:main-vN}.}
Let  $(p_\alpha)_{\alpha\in\Lambda}$ and $(q_j)_{j\in J}$ be two families (possibly but not necessarily the same) of pairwise orthogonal atomic projections, both with sum one. For $C\subset \Lambda$ let $p_C$ have the same meaning as in Proposition~\ref{p:rozklad}. For $D\subset J$ let $q_D$ have the analogous meaning. Let us find the decomposition 
$\Lambda=\bigcup_{\gamma\in\Gamma}Z_\gamma$ using Proposition~\ref{p:rozklad}. Let us further find the analogous decomposition $J=\bigcup_{\delta\in\Delta}V_\delta$ (using the same proposition).
By Proposition~\ref{p:rozklad} there is a bijection $\theta:\Gamma\to \Delta$ such that $p_{Z_\gamma}=q_{V_{\theta(\gamma)}}$ for $\gamma\in\Gamma$.
Therefore we can assume that $\Delta=\Gamma$ and that we have
$$M=\left\{T\in L(H)\setsep T=\sum_{\gamma\in\Gamma}q_{V_\gamma}Tp_{Z_\gamma}\mbox{ in WOT}\right\}.$$
Note that using the two different orthonormal bases we can represent $M$
as
$$M=\left\{T\in L(\ell^2(\Lambda),\ell^2(J))\setsep T=\sum_{\gamma\in\Gamma}q_{V_\gamma}Tp_{Z_\gamma}\mbox{ in WOT}\right\}.$$

Now we are ready to complete the proof using Lemma~\ref{L:Kdiag} and Lemma~\ref{L:1-compl}:

For $T\in K(\ell^2(\Lambda),\ell^2(J))$ define 
$$\Phi(T)=\sum_{\gamma\in\Gamma}q_{V_\gamma}Tp_{Z_\gamma}.$$
By Lemma~\ref{L:Kdiag}(a) the series converges unconditionally in the norm and, moreover, $\Phi$ is a norm-one projection on $K(\ell^2(\Lambda),\ell^2(J))$. Then $\Phi^*$ defines a norm-one projection on $N(\ell^2(J),\ell^2(\Lambda))$ and $\Phi^{**}$ a  norm-one projection on 
$L(\ell^2(\Lambda),\ell^2(J))$. Moreover, if $T\in K(\ell^2(\Lambda),\ell^2(J))$ and $S=\sum_{n=1}^\infty \ip{\cdot}{y_n} x_n\in N(\ell^2(J),\ell^2(\Lambda))$, then
$$\begin{aligned}
\ip{\Phi^*(S)}{T}&=\ip{S}{\Phi(T)}=
\sum_{n=1}^\infty \ip{\Phi(T)x_n}{y_n}
= \sum_{n=1}^\infty \ip{\sum_{\gamma\in\Gamma}q_{V_\gamma}Tp_{Z_\gamma}x_n}{y_n}
\\&=\sum_{n=1}^\infty \sum_{\gamma\in\Gamma}\ip{Tp_{Z_\gamma}x_n}{q_{V_\gamma}y_n} =\sum_{n=1}^\infty \sum_{\gamma\in\Gamma}
\ip{\ip{\cdot}{q_{V_\gamma}y_n}p_{Z_\gamma}x_n}{T}
\\&=\ip{\sum_{\gamma\in\Gamma}\sum_{n=1}^\infty\ip{\cdot}{q_{V_\gamma}y_n }p_{Z_\gamma}x_n}{T}=\ip{\sum_{\gamma\in\Gamma}p_{Z_\gamma}Sq_{V_\gamma}}{T},
\end{aligned}
$$
hence $\Phi^*(S)=\sum_{\gamma\in\Gamma}p_{Z_\gamma}Sq_{V_\gamma}$ for $S\in N(\ell^2(J),\ell^2(\Lambda))$. Note that the series converges absolutely in the operator norm by Lemma~\ref{L:Kdiag}(b). Finally, the same computation shows that
$$\Phi^{**}(T)=\sum_{\gamma\in\Gamma}q_{V_\gamma}Tp_{Z_\gamma}, \quad T\in
 L(\ell^2(\Lambda),\ell^2(J)).$$
So, $\Phi^{**}(L(\ell^2(\Lambda),\ell^2(J)))=M$, therefore $M_*$ can be canonically identified with $\Phi^*(N(\ell^2(J),\ell^2(\Lambda))$. Since this is a $1$-complemented subspace of  $N(\ell^2(J),\ell^2(\Lambda))$, the measures of weak non-compactness with respect to $M_*$ and with respect to $ N(\ell^2(J),\ell^2(\Lambda))$ coincide by Lemma~\ref{L:1-compl}, hence we conclude using the formulas from Theorem~\ref{T:main-N}(b).\qed

\begin{remark} Proposition~\ref{p:rozklad} is a more precise version of the representation result mentioned in \cite[p. 3]{Akemann-Anderson}. A large part of the proof can be done using some results given in \cite{KR2}. Firstly, the assignment $p_\alpha\mapsto p_{Z_\alpha}$ 
corresponds to the assignement from \cite[Proposition 6.4.3]{KR2}. Secondly,
the fact that $p_{Z_\gamma}M$ is canonically isomorphic to $L(p_{Z_\gamma}H)$ is related to the combination of \cite[Proposition 6.4.3]{KR2} (which yields that $p_{Z_\gamma}M$ is a factor of type I) with \cite[Proposition 6.6.1]{KR2} (which says that a type I factor is $*$-isomorphic to $L(H)$ for a Hilbert space $H$). Anyway, for the proof of our result we needed an explicit and canonical version of these representation results.
\end{remark}

\begin{remark}
Theorem~\ref{T:main-vN} provides a partial answer to Problem~5. 
The general case remains open. A solution would require a different method, since by \cite[Theorem 4.2.2]{blackadar} the only von Neumann algebras which are $1$-complemented in $L(H)$ by a weak$^*$-to-weak$^*$ continuous projection are the atomic ones.
\end{remark}

\def\cprime{$'$} \def\cprime{$'$}

\end{document}